\documentclass{amsart}
\usepackage{amsmath,amssymb,amscd,amsfonts}
\usepackage{euler}
\usepackage{epsfig, graphicx,pinlabel}
\usepackage{charter}

\newtheorem{theorem}{Theorem}[section]
\newtheorem{lemma}[theorem]{Lemma}
\newtheorem{proposition}[theorem]{Proposition}
\newtheorem{corollary}[theorem]{Corollary}

\theoremstyle{definition}

\newtheorem{remark}[theorem]{Remark}

\newtheorem{definition}[theorem]{Definition}

\newtheorem*{ack}{Acknowledgement}

\def\til{\widetilde}
\def\mbb{\mathbb}

\def\mcl{\mathcal}

\def\ten{\otimes}

\def\strl{\stackrel}
\def\tu{\textup}

\def\a{\alpha}

\def\d{\delta}

\def\e{\epsilon}

\def\sm{\sigma}

\def\o{\omega}

\def\bP{\mbb P}

\def\bG{\mathbb G}
\def\bC{\mathbb C}
\def\inj{\hookrightarrow}

\def\deg{\textup{deg} \, }

\def\spec{\tu{Spec\,}}
\def\proj{\tu{Proj  }}

\def\hilb{Hilb}

\def\Mg{\overline{M}_g}

\def\FMg{\overline{\bf \mathcal M}_g}

\def\inj{\hookrightarrow}

\def\M{\bar{M}}
\def\cO{\mathcal O}

\def\bar{\overline}
\def\isomto{\strl{\sim}{\to}}

\def\tac0{\textup{Tac}_0}
\def\Hhs{\overline H_4^{hs}}
\input xy
\xyoption{all}

\input epsf
\def\Mhs{\bar M_4^{hs}}
\def\dhs{\delta^{hs}}
\def\bQ{\mathbb Q}
\def\chow41{Chow_{4,1}/\!\!/SL_4}
\def\chow{Chow}
\def\tilde{\widetilde}

\def\Mg{\bar M_g}
\def\M{\bar M}
\def\SL{\textup{SL}}

\begin{document}

\title[Contraction of genus two tails]{A birational contraction of genus two tails in the moduli space of genus four curves I}



\author{Donghoon Hyeon}
\address{Department of Mathematics\\ POSTECH \\ Pohang, Gyeongbuk
790-784 \\ Republic of Korea\\}
\email{dhyeon@postech.ac.kr}

\author{Yongnam Lee}
\address{Department of Mathematical Sciences \\ Korea Advanced Institute of Science and Technology \\  Daejon 305-701 \\
Republic of Korea \\ and Korea Institute for Advanced Study \\ Seoul 130-722 \\ Republic of Korea \\}
\email{ynlee@kaist.ac.kr}

\subjclass[2000]{Primary 14L24, 14H10 \\Secondary 14D22}
\keywords{moduli, stable curve, log canonical model, tail}

\begin{abstract} We show that for $\alpha \in (2/3, 7/10)$, the log canonical model $\bar M_4(\alpha)$ of the pair $(\bar M_4, \alpha \delta)$ is isomorphic to the moduli space $\bar M_4^{hs}$ of h-semistable curves constructed in  \cite{HH2}, and that there is a birational morphism $\Xi: \bar M_4^{hs} \to \bar M_4(2/3)$ that contracts the locus of curves $C_1\cup_p C_2$ consisting of genus two curves meeting in a node $p$ such that $p$ is a Weierstrass point of $C_1$ or $C_2$. To obtain this morphism, we construct a compact moduli space $\M_{2,1}^{hs}$ of
pointed genus two curves  that have nodes, ordinary cusps and
tacnodes as singularities, and  prove that it is isomorphic to
Rulla's flip constructed in \cite{Rulla}.
\end{abstract}

\maketitle

\section{Introduction}

In this paper, we continue our investigation of the log minimal model program for
the moduli space of stable curves. In \cite{Has}, \cite{HH1}, \cite{HH2},
\cite{HL2}, \cite{HL3}, Brendan Hassett and the authors have
studied the log canonical models
\[
\bar M_g(\a) = \proj \oplus_{n \ge 0} \Gamma(\bar M_g, n(K_{\FMg} + \a\d)), \quad \a \in (7/10-\e, 1]
\]
where $\e$ is a small positive rational number. Each log canonical
model was given a GIT construction and interpreted as a moduli
space, and their relations were concretely described.
For $g \ge 3$, the first divisorial contraction $T$ occurs at $\a
= 9/11$ followed by a flip at $\a = 7/10$. \small\[
\xymatrix{ \bar M_g(1)\simeq\M_g \ar[d]_-{T} &&&  \\
\bar M_g(\frac9{11})\simeq  \bar M_g^{ps} \ar[ddr]_-{\Psi} &
& \bar M_g(\frac7{10}-\e) \simeq \bar M_g^{hs}\ar[ddl]^-{\Psi^+} \\
\\
& \bar M_g(\frac7{10})\simeq \bar M_g^{cs} &  \\
 }
\]\normalsize
Here $\Mg^{ps}$, $\Mg^{cs}$ and $\Mg^{hs}$ respectively denote the moduli spaces
of pseudostable curve \cite{Sch}, of c-semistable curves and of
h-semistable curves \cite{HH2}.
 In this paper, we shall prove that for $g = 4$, the sequence of maps above is followed by another small contraction:
\begin{theorem}\label{T:1}
 \begin{enumerate}
\item $\M_4(\a)$ is isomorphic to $\M_4^{hs}$ for $2/3 < \a < 7/10$.
\item There is a birational morphism
$
\Xi : \Mhs \to \M_4(2/3)
$
contracting the locus of Weierstrass genus two tails.
\end{enumerate}\end{theorem}
A {\it  tail} of a curve $C$ is a subcurve
that meets the rest of the curve in one node.  It is called a {\it
Weierstrass tail} if moreover the subcurve is hyperelliptic and the attaching node is
fixed under the hyperelliptic involution. Abusing terminology, we
will call $C$ itself a (Weierstrass) genus two tail if it has a
(Weierstrass) genus two tail as a subcurve. Our proof uses the
results from \cite{CMJL} and \cite{Kim} on the GIT of canonically embedded
 genus four curves  and thus applies only
to the case $g = 4$, but  we believe that the statement is true
for any $g \ge 4$. In fact, one should be able to remove  the use of \cite{CMJL,Kim} by
constructing the GIT quotient space of {\it $6$-Hilbert semistable
curves}. Examination of the slope of the linearization  suggests
that $\M_g(2/3)$ is isomorphic to the GIT quotient of the Hilbert
scheme of the sixth Hilbert points of bicanonical curves. More
precisely, $K_{\bar{\mathcal M}_g} + \frac23 \d$, when  pulled
back to  the Hilbert scheme of $\nu$-canonical curves, is
proportional to the linearization \cite[Equation~(5.3)]{HH2} when
$\nu = 2$ and $m = 6$. But we have not been able to construct the
GIT quotient mainly because of our lack of understanding of the finite
Hilbert stability of curves. Recently, Ian Morrison and Dave
Swinarski made a breakthrough \cite{MS} constructing examples of small
genera curves that are $6$-Hilbert stable.
 Building upon this
work, we plan to carry out the GIT and give modular
interpretations of $\bar M_g(2/3)$ and the small contraction $\Xi$
for any $g \ge 4$.

To prove our main theorem, we also construct a new moduli space
$\M_{2,1}^{hs}$ of pointed curves of genus two that allows nodes,
ordinary cusps and tacnodes as singularities, and prove that it is
isomorphic to the flip of $\M_{2,1}$ given in
\cite[Section~3.9]{Rulla}\footnote{Rigorously what we want to say here
is the flip of $\M_{2,1}^{ps}$, but it will be clear from the context what we mean and there is no need to introduce a new moduli  space.}. It is constructed by using GIT, and the
method can be generalized to give new compactifications of
$M_{g,n}$ for any $g \ge 2$ and $n$. We hope to carry this out in a future work.


After this work appeared on the arXiv, several exciting results have been announced including the aforementioned \cite{CMJL} where the GIT of Chow point of canonical genus four curves is completely worked out via the GIT of cubic fourfolds. They show that the GIT quotient space is  isomorphic to $\M_4(5/9)$. In a recent article \cite{CMJL12}, they also obtain a complete list of log canonical models $\M_4(\a)$ for $\a \le 5/9$ (which ends with the final model $\M_4(29/60)$ worked out by Fedorchuk \cite{Fedorchuk}) by considering the GIT of canonical images as a complete intersection of a quadric and a cubic in $\bP^3$.
 Morrison and Swinarski's breakthrough work on finite Hilbert stability was followed by
\cite{AFS} where Alper, Fedorchuk and Smyth prove that a generic smooth curve, canonically or bicanonically embedded, has a semistable
$m$-th Hilbert point for all
$m\ge 2$. Their starting point is Kempf-Morrison theory after Morrison-Swinarski, but they work out a suitable basis for the (bi)canonical system by hand. By using an entirely different,
simple geometric argument, Fedorchuk and Jensen prove the GIT semistability of the 2nd Hilbert point of Gieseker-Petri general canonical curves \cite{FJ}.
The last step of the log minimal model program for $\M_4$ has recently been completed by Fedorchuk \cite{Fedorchuk} where he proves that $\M_4(\a)$ is isomorphic to the GIT quotient of the linear systems of $(3,3)$ curves on $\bP^1\times \bP^1$ for $8/17 < \a \le 29/60$, and that the birational map $\M_4 \dashrightarrow \M_4(29/60)$ contracts $\Delta_1, \Delta_2$ and the Petri divisor.


\begin{ack} We would like to thank Brendan Hassett for
sharing his ideas which run through this work, especially \S\ref{SS:delta2}. We greatly benefited from conversations with Dave Swinarski, from whom we learned the computation in \S\ref{S:polarization} among other things about the GIT of pointed curves. We also thank the anonymous referees who very carefully examined the article, pointed out errors and gave numerous suggestions.
The
first author was partially supported by the following grants funded by the government of Korea:
the Korea Institute for
Advanced Study (KIAS) grant,
NRF grant 2011-0030044 (SRC-GAIA), and NRF grant 2010-0010031. The second author was supported
by the NRF funded by the Korea government(MEST) (No. 2010-0008752).

\end{ack}

\section{h-stable curves and pointed curves}\label{S:h-stable}
In \cite{HH2}, Hassett and the first named author classified the bicanonically embedded curves of
genus $g \ge 4$ whose $m$th Hilbert point is GIT (semi)stable for $m
\gg 0$. In the genus four case (and this case only), semistable points are stable, and those are the {\it h-stable} curves $C$ characterized by the following properties:
\begin{enumerate}
\item $C$ has nodes, ordinary cusps and tacnodes as
    singularities;
\item a smooth rational component of $C$ meets the rest of the curve
    in $\ge 3$ points counting multiplicity;
\item a genus one subcurve of $C$ meets the rest of the curve
    in $\ge 3$ points counting multiplicity;
\item no {\it elliptic chain} is admitted.
\end{enumerate}

It  follows from this that the moduli stack $\bar{\mcl
M}_4^{hs}$ of h-stable curves of genus four is Deligne-Mumford and
the coarse moduli space $\bar M_4^{hs}$  has only finite quotient
singularities. We let $\dhs$ denote the divisor  of singular
curves in the moduli stack. It has two irreducible components
$\dhs_{irr}$ consisting of irreducible curves and $\dhs_2$
consisting of genus two tails. Since the morphism $\bar {\mcl
M}_4^{hs} \to \bar M_4^{hs}$ is unramified (as opposed to $\overline{\mathcal M}_g
\to \Mg$ which is ramified along $\delta_1$), we shall abuse
notation and use $\dhs_{irr}$ and $\dhs_2$ to denote the corresponding divisor classes in the
moduli space. In \cite[Theorem~9.1]{HH2}, it is shown that the
linearization on the Hilbert scheme used in the GIT quotient
construction of $\Mhs$ descends to (a positive rational multiple
of)
\[
K_{\Mhs} + (7/10 - \e) \dhs
\]
on $\Mhs$ for some small positive rational number $\e$. Thus
$K_{\Mhs} + (7/10 - \e) \dhs$ is an ample $\bQ$-divisor on $\Mhs$. By using elementary properties of global sections of line bundles under birational maps
(see for instance, \cite[Lemma~2.2, Proposition~2.3]{Simpson}),  we conclude that:
\begin{lemma} There is a natural isomorphism
\[
\bar M_4(2/3) \simeq Proj \oplus_{n \ge 0} \Gamma\left(\bar M_4^{hs}, n\left(K_{\bar M_4^{hs}} + \frac23 \dhs\right)\right).
\]
\end{lemma}
We shall therefore study the map associated to $K_{\bar M_4^{hs}}
+ 2/3 \dhs$. An overarching principle of the Hassett-Keel program is that the
intermediate maps contract the boundary divisors $\Delta_1, \Delta_2,
\dots$ one by one (up to the point where $\Delta_i$ is of general type), so we seek extremal rays for $K_{\bar M_4^{hs}}
+ \frac23 \dhs$ generated by a curve in the boundary divisor
$\dhs_2$. For this purpose, we shall use  the degree two finite map
\begin{equation}\label{E:m21tom4}
\bar M_{2,1} \times \bar M_{2,1} \to \Delta_2 \inj \bar M_4
\end{equation} and pass to $\bar M_{2,1}$, whose birational geometry has been
extensively studied in \cite{Rulla}. But to use Rulla's results in
studying the birational geometry of $\bar M_4^{hs}$, we need to
have a map for $\bar M_4^{hs}$  analogous to (\ref{E:m21tom4})
above, which breaks $\dhs_2$ into two components. Since such component curves may have an ordinary cusp
or a tacnode, we need to construct a compact moduli space of
pointed curves allowing these singularities. This is the task we
carry out in  \S\ref{S:pointed-hs}. Here, we sketch the
construction and summarize the main properties of the new moduli
space.

\subsection{GIT setup and the stability condition}
Let $ C_1\cup_pC_2$ be a curve in $\dhs_2$, where $C_i$ are of
genus two. Restricting the polarization $\omega_C^{\otimes 2}$ to $C_i$ yields
$\omega_{C_i}^{\otimes 2}(2 p)$, which leads us to consider the
GIT of pointed genus two curves $(C, p)$  embedded in $\bP^4$ by
the complete linear system $|\omega_C^{\otimes 2} (2p)|$. We say
that $(C,p)$ is \textit{$\nu$-log-canonically embedded} if $C$ is
embedded in $\bP^N$ by $|\omega_C^{\otimes \nu}(\nu p)|$, $N+1 =
(2\nu-1)(g-1) + \nu$. In our  $\nu = 2$ case, $(C, p)$ is said to
be {\it bi-log-canonically embedded}. We recall the basic GIT setup
from \cite{Swin}. Let $d  = 6$, $N + 1 = 5$, and $P(m) = 6m - 1$.
Let $\hilb$ be the Hilbert scheme parametrizing the subschemes of
$\bP^N$ whose Hilbert polynomial is $P(m)$. Let $J$ denote the
subscheme of $\hilb \times \bP^N$ consisting of  the points
associated to bi-log-canonical pointed curves $(C, p)$ such that $C$ is
smooth and $p \in C$ is a smooth point.
The parameter space we use is
$\bar J$, the closure of $J$ inside $\hilb \times \bP^N$,
linearized by
\begin{equation}\label{E:balanced}
\mcl L_m := 3 \det \pi_*\cO_{\mcl C}(m) + 2m^2
 \sigma^*\cO_{\mathcal C}(1)
\end{equation}
where $\pi : \mcl C \to \bar J$ is the universal curve and $\sigma : \bar J \to \mathcal C$, the universal section. Here, the
balancing factor $2m^2/3$ was chosen so that the GIT quotient
space parametrizes exactly the pointed curves that we desire. Such
a balancing factor was employed in \cite{BS} and \cite{Swin}, and
for $\nu$-log-canonical pointed curves, it turns out that $\nu m^2/ (2\nu -1)$
works. We say that $(C,p)$ is {\it
Hilbert $($semi$)$stable} if the corresponding point in $\bar J$ is
GIT (semi)stable with respect to the standard $SL_{N+1}$ action
and
 $\mcl L_m$ for all $m \gg 0$.

\begin{theorem}\label{T:hs} A bi-log-canonical pointed curve $(C, p)$ of genus two is Hilbert stable if and only if it is  {\it h-stable} i.e. it satisfies the following properties:
\begin{enumerate}
\item $C$ has nodes, ordinary cusps and tacnodes as
    singularity;
\item A rational component of $C$ has $\ge 3$ special points
    counting multiplicity;
\item $C$ does not have a subcurve of genus one.
\item $p$ is a smooth point.
\end{enumerate}
Here, special points mean singular points or marked points.
\end{theorem}
We will prove this theorem in \S\ref{S:pointed-hs}. We shall let $\M_{2,1}^{hs}$ denote the moduli space of Hilbert stable bi-log canonical pointed curves.

\subsection{Product decomposition of the boundary}
Since h-semistable pointed curves have finite automorphisms, there is no strictly
semistable point and the quotient space $\bar M_{2,1}^{hs}$ has
finite quotient singularities. Consequently the corresponding
moduli stack is Deligne-Mumford. As in the case of $\bar{\mathcal M}_4$, we have the morphism
\[
\bar {\mcl M}_{2,1}^{hs} \times \bar {\mcl M}_{2,1}^{hs}
 \to \bar {\mcl M}_{4}^{hs}
 \]
of stacks given by attaching two h-stable pointed curves at the
marked points forming a node. This leads to the following straightforward
generalization of \cite[Lemma~1.1]{GKM}. This follows from the line bundle decomposition formula \cite{faber-1997}, and it is straightforward to check that the formula also works for the h-semistable case.

\begin{lemma}\label{L:slice} An effective curve $C$ in the product $\bar M_{2,1}^{hs} \times \bar  M_{2,1}^{hs}$ is numerically equivalent to the sum  of curves $C_1\times\{pt\} + \{pt\}\times C_2$ where $C_i$ is obtained by projecting $C$ to the $i$th factor. In particular, if a line bundle on the product is nef on each slice $\bar M_{2,1}^{hs} \times pt$ and $pt \times \bar M_{2,1}^{hs}$ of the product, then it is nef.
\end{lemma}

\subsection{Polarization on the moduli space}\label{S:polarization}
The linearization on $\bar J$ descends to an ample line bundle on
the GIT quotient of $\bar J$ by a suitable special linear group, which can be written in terms
of tautological classes on the quotient space. Recall that
$\pi : \mcl C \to \bar J$ is the universal curve and let $\psi$ denote
the bundle of cotangent lines. In the polarization formula below,
the first term $\det \pi_*\cO_{\mathcal C}(m)$ comes from the
curve part and the second, the point part. We follow \cite[Page
106]{Mum} and write $\det \pi_*\cO_{\mathcal C}(m)$ in terms of
the Hodge class $\lambda$, the boundary divisor $\delta_{irr}$ of
irreducible curves and the bundle $\psi$ of cotangent lines. Mumford's formula can be modified without too much difficulty to give a polarization formula for the pointed case, from which we get
\[
\begin{array}{lll}
\det \pi_*\cO_{\mathcal C}(m) + \frac23m^2 Q
 =  (2m^2-m)(\tilde{\kappa} + \psi) + \lambda + (m(6m-1) + \frac23m^2) Q \\
\end{array}
\]
where $Q = -\frac 15 (\tilde \kappa + \lambda + \psi)$ satisfying $\cO_{\mathcal C}(1) = \omega_{\mathcal C}^{\otimes 2}(2\sigma) \otimes \pi^*Q$. Using the
relations
\[
\begin{array}{c}
\delta_1 = 5\lambda - \frac12 \delta_{irr}\\
\tilde \kappa = 12 \lambda - \delta = 7\lambda - \frac12 \delta_{irr}\end{array}
\]
we may rewrite $Q = -\frac15 (8\lambda + \psi - \frac 12
\delta_{irr})$ and obtain \small

\begin{equation}\label{E:polform}
\left(\frac{10}3 m^2 - \frac{27}5 m + 1\right) \lambda + \left(\frac23 m^2 - \frac45 m\right)\psi - \left(\frac13 m^2 - \frac25m\right)\delta_{irr}
 \sim (10 - \e(m)) \lambda + 2\psi - \delta_{irr}
 \end{equation}
\normalsize
where $\epsilon(m) = \frac{21m - 150}{100 m^2 - 120 m}$ tends to
$0$ as $m \to \infty$.

\vskip .2in

The results summarized in this section will be crucial in the proof of the main Theorem~\ref{T:1}  in the subsequent section.

\section{Construction of the birational morphism $\Xi$}
We shall first prove that $|\ell (K_{\Mhs} + 2/3 \dhs)|$ is
basepoint free for sufficiently large and divisible $\ell$. Let $H' = (\Psi^{+})^{-1}(\overline H_4^{cs})$ where $\Psi^{+} : \M_4^{hs} \to \M_4^{cs}$ is the small contraction that is the flip to $\Psi : \M_4^{ps} \to \M_4^{cs}$ \cite[Theorem~2.12]{HH2} and $\overline H_4^{cs}$, the closure in $\M_4^{cs}$ of the locus of smooth hyperelliptic curves.

We
shall accomplish this by showing that
\begin{itemize}
\item[(i)]  $K_{\Mhs} + 2/3 \, \dhs$ is nef on $H'$.
\item[(ii)]   $|\ell(K_{\Mhs} + 2/3 \, \dhs)|$ gives rise to an
    isomorphism on $\Mhs\setminus (\dhs_2\cup H')$;
\item[(iii)] $K_{\Mhs} + 2/3 \, \dhs$ is nef on $\dhs_2$;

\end{itemize}
Since (i) $\sim$ (iii) together imply that $ K_{\Mhs} + 2/3 \,
\dhs$ is nef, the basepoint freeness follows by Kawamata-Viehweg
basepoint freeness theorem \cite[3.3]{KM}. To apply the basepoint freeness theorem, we need to check that the pair $(\Mhs, 2/3\, \dhs)$ is klt. For $g=4$, one can use the deformation theory worked out in \cite[Proposition~6.1, Corollary~6.3]{HH2}
and follow the proof of \cite[Theorem~5.2]{DM} to show that that the moduli stack of h-stable curves is smooth, from which it follows that $(\Mg^{hs}, \a\dhs)$ is smooth for $0\le \alpha < 1$. But, instead of providing details for this argument involving the moduli stack of h-stable curves, we shall give an alternate proof that works in all genera.

First, $(\Mg, \a\d)$ is klt for $0 \le \alpha < 1$ since the moduli stack $\overline{\mathcal M}_g$ is smooth and $\delta$ is normal crossings. See \cite[Proposition~A.9, A.13]{HH1}. Recall from \cite{HH2} that running the log MMP produces a divisorial contraction $T: \Mg \to \Mg^{ps}$, a small contraction $\Psi : \Mg^{ps} \to \Mg^{cs}$ and its flip $\Psi^+: \Mg^{hs}\to \Mg^{cs}$ where $\Mg^{ps}$ and $\Mg^{cs}$ are the moduli spaces of pseudo-stable curves and of c-semistable curves, respectively. We point the readers to \cite{HH2} for definitions and details. For our purpose here, we are only concerned with the fact that
\[
\begin{array}{rcl}
(\Psi^+)^*(K_{\Mg^{cs}} + 7/10\, \d^{cs}) & = & K_{\Mg^{hs}}+ 7/10\, \d^{hs}  \\
\Psi^*(K_{\Mg^{cs}} + 7/10\, \d^{cs}) & = & K_{\Mg^{ps}}+ 7/10\, \d^{ps}
\end{array}
\]
and the log discrepancy formula
\[
T^*(K_{\FMg^{ps}} + \alpha \, \d^{ps}) = K_{\FMg} + \alpha \, \d + (11\alpha - 9) \, \d_1.
\]
The log discrepancy formula in conjuction with the previously observed fact that $(\Mg, \alpha \,\d)$ is klt for $0 \le \alpha <1$  implies that $(\Mg^{ps}, 9/11 \, \delta^{ps})$ is klt (and hence so for $0 \le \alpha \le 9/11$ and in particular, for $\alpha = 7/10$). Subsequently, \cite[Lemma~3.10]{Kollar} imply that $(\Mg^{hs}, 7/10 \, \d^{hs})$ is klt, and therefore the pair is klt for $0 \le \alpha \le 7/10$.

\

\subsection{$K_{\Mhs}+2/3\, \dhs$ is nef on $H'$} It is shown in \cite[Section~3.3]{HL3} that $K_{\M_4^{cs}} + \alpha \delta^{cs}$ pulls back to a line bundle that is ample on $\overline H_4(7/10)$ for $2/3 < \alpha \le 7/10$. There is an error in \cite[Section~3.3]{HL3} but the results in Section 3.3 and \cite[Theorem~1.(2)]{HL3} are valid for $g = 4$. We will briefly explain this and justify our use of the results in this paper. The error is in the assertion that $L_\alpha +\frac{11\a-9}2 \widetilde B_3 + (10\alpha - 7) \widetilde B_4$ is nef for $\alpha = 2/3$. Here, $L_{\alpha}$ denotes the restriction (pull-back) of the log canonical divisor $K_{\FMg} + \a\d$ to $\overline H_g$, and $L_\alpha +\frac{11\a-9}2 \widetilde B_3 + (10\alpha - 7) \widetilde B_4$ is the pullback of the cycle image $L^{''}_{\alpha}$ of $L_\alpha$ in $\overline H_4(7/10)$ under $\overline H_4 \to \overline H_4(7/10)$ \cite[Lemma~2]{HL3}.
 In general, $L^{''}_\alpha$  intersects the F-curves $\{1,1,2s+1,2g-2s-1\}$ negatively for $2 \le s \le \lfloor g/2 \rfloor$ and for $\alpha < 7/10$. These F-curves correspond to the disconnecting elliptic bridges.
When $g = 4$, for genus reasons the only disconnecting elliptic bridge is the F-curve $\{1, 1, 3, 5\}$ ($s=1$ case), and the intersection of $L^{''}_\alpha $ with this F-curve is
\[
5\alpha - \frac72 + \frac{11\alpha-9}2(-1) +(10\alpha-7)\cdot 2 = 13\left(\frac32 \alpha - 1\right).
\]
Hence, as asserted in \cite[Section~3.3]{HL3}, $L^{''}_{\alpha}$ is ample on $\overline H_4(7/10)$ for $2/3 < \alpha \le 7/10$ and nef for $2/3 \le \alpha \le 7/10$.
It follows that $\overline H_4(\a)$'s are isomorphic to $\overline H_4(7/10)$  for $2/3 < \a < 9/11$, and by \cite[Proposition~1]{HL3}, $\overline H_4(7/10)$ is isomorphic to $\overline H_4^{cs}$. Since $K_{\M_4^{cs}} + 2/3 \delta^{cs}$ is a positive rational multiple of $L^{''}_{2/3}$, it is nef on $ \overline H_4^{cs}$, and subsequently $K_{\Mhs} + 2/3 \dhs = (\Psi^+)^*(K_{\M_4^{cs}} + 2/3 \delta^{cs})$ is nef on $H'$.
Why is $\{1,1,3,2g-3\}$ different from the other F-curves? What happens is, since $\{1,1,3,5\}$ is an elliptic tail, the image $C_{1,1,3}$ of this F-curve under the first divisorial contraction $T: \M_4 \to \M_4^{ps}$ is a union of a genus two curve and another with an ordinary cusp attached nodally at a Weierstrass point. In particular, it is no longer an elliptic bridge, and the locus of these curves misses the extremal loci of the first small contraction and of its flip.
Abusing notation, we can think of $C_{1,1,3}$ as a curve in $\delta_2^{cs} \subset \M_4^{cs}$ (or $\delta_2^{hs} \subset \M_4^{hs}$). We shall prove in Section~\ref{SS:delta2} that  $K_{\Mhs}+2/3\, \dhs$ is nef on $\dhs_2$, which is consistent with the results in this section as far as $C_{1,1,3}$ is concerned.

\

By considering the c-semistable replacements of the stable hyperelliptic curves, we conclude that $H'$ has two components: the strict transformation $\overline H_4^{hs}$ of $\overline H$ in $\Mhs$, and the locus $\tac0$ consisting of h-stable curves $C$ with a tacnode $p$ normalized by $\nu : D \to C$, $\nu^{-1}(p) = \{p_1, p_2\}$, such that $p_i$ are interchanged by the hyperelliptic involution of $D$. The argument above shows that $K_{\Mhs} + 2/3 \dhs$ non-negatively intersects with any curve in $H'$, whether it lies in $\Hhs$ or $\tac0$. Indeed, the only difference between members in $\Hhs$ and those in $\tac0$ is the tangent space identification data for the tacnode, which is not visible in $\M_4^{cs}$ i.e. any two tacnodal curves with the same (partial) normalization of the tacnode are identified in $\M_4^{cs}$ (see \cite[Lemma~9.2, Proposition~9.6 and its proof]{HH2}). Hence a curve $Z \subset \tac0$ maps into $\overline H_4^{cs}$ on which $K_{\M_4^{cs}}+2/3\delta^{cs}$ is nef, so $K_{\M_4^{hs}}+2/3 \delta^{hs} = (\Psi^+)^*(K_{\M_4^{cs}}+2/3\delta^{cs})$ non-negatively intersects $Z$.

Yet another way to see this is that any curve in $\tac0$ can be deformed to a curve in $\Hhs$ with base a smooth rational curve. To fully justify this argument should require some work, but we sketch the idea how it goes.
Given a hyperelliptic curve $D$ and two conjugate smooth points $p_1, p_2$, there is a unique tacnodal curve $(C, p) \in \Hhs$ normalized by $(D, p_1,p_2)$, and any other tacnodal curve normalized by $(D,p_1,p_2)$ can be deformed to $(C,p)$ just by varying the tangent space identification used to form the tacnode. This can be done for families too. A family of curves $(\mathcal C, \tau) \to B$ in $\tac0$ is obtained from a family $\varpi: \mathcal D \to B$ of genus two curves, two conjugate sections $\sm_1, \sm_2$, and a tangent space identification data $\gamma(t) : T_{\sm_1(t)}\mathcal D_t \isomto T_{\sm_2(t)}\mathcal D_t$ used to form the tacnode $\tau : B \to \mathcal C$. Let $\psi_i := \sigma_i^*(\omega_\varpi)$ be the bundle of cotangent lines over $B$ along $\sm_i$ i.e. $\psi_i|_t = \omega_{\mathcal D_t}|_{\sm_i(t)}$ for any $t \in B$. Then $\gamma$ is just a nowhere vanishing section of $\psi_1\otimes \psi_2^*$.
And of course, in the space of nowhere vanishing sections of the line bundle $\psi_1\otimes \psi_2^*$, we can deform $\gamma$ to the special $\gamma^\star$ which gives rise to the unique curve in $\tac0$.

\subsection{ $K_{\Mhs} + 2/3 \, \dhs$ gives rise to an isomorphism
over $\Mhs\setminus (\dhs_2\cup H')$.}\label{S:isom} Let $\e$ be a
small positive rational number. Note that we can write
\begin{equation}\label{E:lincomb}
K_{\Mhs} + 2/3 \, \dhs = a \left(K_{\Mhs} + (7/10 -\e) \, \dhs\right)
+ b \left(K_{\Mhs} +  5/9 \, \dhs\right)
\end{equation}
with positive coefficients $a = 10/(13 - 90 \e)$ and $b = (3 - 90\e) / (13 - 90\e)$. Since $K_{\Mhs} + (7/10 -\e) \, \dhs$ is ample for small $\e$, it
suffices to show that $K_{\Mhs} +  5/9 \, \dhs$ is generated by sections away from $ \dhs_2\cup H'$. For this end, we exploit the
moduli space $\chow_{4,1}/\!\!/SL_4$ of Chow semistable canonical
curves of genus four.  This space was carefully studied by Casalaina-Martin-Jensen-Laza \cite{CMJL}, and H. Kim
\cite{Kim} also achieved some partial results, including
\begin{theorem}\label{T:CMJL}\cite[Theorem~3.1]{CMJL} A complete connected curve of genus four
has a Chow semistable (resp. Chow stable) canonical image if it is
h-stable and not contained in $\dhs_2\cup H'$ (resp. $\dhs_2\cup H'\cup {\rm Tac}$), where ${\rm Tac}$ denotes   the locus of the tacnodal curves.
\end{theorem}

\begin{corollary} There is a birational map $\Theta: \Mhs \dashrightarrow \M_4(5/9)$ that is regular away from $\dhs_2\cup H'$.
\end{corollary}

Since $\Mhs$ has only finite quotient singularities, it is normal. In fact, $\Mg^{hs}$ is normal in general since it is a GIT good quotient of the locus of h-stable curves in the Hilbert scheme which is smooth  by \cite[Proposition~6.1, Corollary~6.3]{HH2}. Likewise, $\M_4(5/9)$ is normal \cite[p18]{CMJL}. It follows that $\Theta$ extends to a birational map regular away from a codimension two locus, and $\Theta^*\mathcal O(+1)$ is well defined. In the proof of \cite[Theorem~4.1]{CMJL}, it is shown that $\Theta^*\mathcal O(+1)$ is a positive rational multiple of
$9\lambda - \delta^{hs} - 2\delta_2^{hs}$ \cite[Theorem~4.1]{CMJL}.

 Now, let $D$ be an h-stable curve  not in
$\dhs_2 \cup H'$. Due to Theorem~\ref{T:CMJL}, $\Theta$ is regular at $D$.  Since $K_{\Mhs} + 5/9 \dhs$ is proportional to $9 \lambda - \dhs$, it follows that there is a section of (a sufficiently large
and divisible multiple of) $K_{\Mhs} + 5/9 \dhs$ that does not
vanish at the point representing $D$. Since $K_{\Mhs} + (7/10 -
\e) \dhs$ is ample, it follows that the linear combination
(\ref{E:lincomb}) gives rise to an isomorphism on $\Mhs \setminus
(\dhs_2\cup H')$.

\subsection{$K_{\Mhs} + 2/3 \dhs$ is nef on
$\dhs_2$.}\label{SS:delta2}

Fix an h-stable pointed curve $(C_o, p_o)$ of genus two and let $j'
: \bar M_{2,1}^{hs} \to \bar M_4^{hs}$ denote the map sending
$(C,p)$ to $C\cup_{p=p_o} C_o$, the h-stable curve obtained by
gluing $C$ and $C_o$ such that $p$ and $p_o$ are identified to a
node. Due to
Lemma~\ref{L:slice}, to establish the nefness of $K_{\Mhs} + 2/3
\dhs$ on $\dhs_2$, it suffices to show that its pullback to $\bar
M_{2,1}^{hs}$ by $j'$ is nef.
 If $j: \bar M_{2,1} \to \bar M_4$ is the corresponding map,
then the commutative diagram of  maps
\[
\xymatrix{
\bar M_{2,1} \ar[r]^-j \ar@{-->}[d] & \bar M_4 \ar@{-->}[d] \\
\bar M_{2,1}^{hs} \ar[r]_-{j'} & \bar M_4^{hs} \\
}
\]
gives rise to a commutative diagram of the Picard groups.  Therefore, to analyze $(j')^*(K_{\Mhs} + 2/3\dhs)$ , we shall consider $j^*((\Psi^+)^{-1}\circ\Psi)^*(K_{\Mhs} + 2/3\dhs)$.
By the discrepancy formula \cite[Lemma~4.1]{HH1}, $K_{\bar
M_4^{hs}} + \alpha \d^{hs}$ corresponds to $K_{\bar{\mcl M}_4} +
(\alpha - 2)\d + (11\a - 9 )\d_1$, which pulls back to
\[
13 \lambda + (\a - 2) (\d_{irr} - \psi)  + (12\a - 11) \d_{1,\{1\}} = j^*((\Psi^+)^{-1}\circ\Psi)^*(K_{\Mhs} + 2/3\dhs).
\]
For $\a = 2/3$, this is a positive rational multiple of
\[
- \d_{irr} - 12 \d_{1,\{1\}} + 40 \psi
\]
which is in turn proportional to
$
D =  3 (- \d_{irr} - 12 \d_{1,\{1\}} + 40 \psi)$  \cite[Proposition~3.3.7]{Rulla}.
Rulla showed that $D$ gives rise to a rational map $\varphi_{\bar
Q}$ on $\M_{2,1}$ which contracts $\Delta_{1,\{1\}}$ and
\[
\bar W_2^1 = \{(C,p) \, | \, \mbox{$p$ is a Weierstrass point}\}
\]
and that $\varphi_{\bar Q}$ becomes regular after a flip
$
\bar M_{2,1} \dashrightarrow Z.
$
Then by the contraction theorem, there is a birational morphism $
Z \to Y$ contracting $\Delta_{1,\{1\}}$, and $D$ remains nef on
$Y$ \cite[Section~3.9]{Rulla}.

\begin{proposition}\label{P:Y} $\bar M_{2,1}^{hs}$ is isomorphic to Rulla's flip $Y$.
\end{proposition}

\begin{proof} Note that  $\bar M_{2,1}^{hs}$ and $Y$ are isomorphic away from a locus of codimension two since both admit
birational maps from $\bar M_{2,1}$  contracting $\Delta_{1,\{1\}}$. It remains to show that an ample line bundle on $\bar M_{2,1}^{hs}$ is ample on $Y$.

We start with $K_{\bar {\mcl M}_4} + \a \delta$ on $\bar M_4$. We
shall prove that
\[
j^*(K_{\bar {\mcl M}_{4}} + \a \delta)  = j^* ( 13 \lambda + (\a - 2) \delta) = 13 \lambda + (\a - 2) (\delta_{irr } + \delta_{1, \{1\}}) + (2-\a) \psi \\
\]
is ample on both $Y$ and on $\bar M_{2,1}^{hs} = \bar J/\!\!/SL_5$
(Theorem~\ref{T:hs}) when $\a = 7/10 - \e$ for some small
positive rational number $\e$. Note that for $\a = 7/10 - \e$, it
is proportional to
\begin{equation}\label{E:L1}
-\e' \delta_{irr} + (13 - \e') \d_{1,\{1\}} + (13 + \e') \psi, \quad \e' = 10\e.
\end{equation}
On $Y$, $\d_{1,\{1\}}$ is contracted and (\ref{E:L1}) can be
written as a positive rational linear combination of $C \sim
\d_{irr} - 18 \d_{1,\{1\}} + 45 \psi$ and $D \sim -\d_{irr} - 12
\d_{1,\{1\}} + 40 \psi$. Thus it is ample on $Y$ since $C$ and $D$
are extremal  on the two dimensional nef cone of $Y$
\cite[\S~3.9]{Rulla}.

In \S\ref{S:h-stable}, Equation (\ref{E:polform}), we found the
linearization $\mcl L_m$  on the GIT quotient $\bar J/\!\!/SL_5$ to
be a positive rational multiple of $ \left( 10 - \e(m)\right)
\lambda - (\d_{irr} - 2 \psi)$ where $\e(m) = \frac{21m - 150}{100
m^2 - 120 m}$. On the other hand,  $$ K_{\overline{\mathcal M}_4}
+ (7/10 - \e)\delta \sim (10 - \e) \lambda - (\delta_{irr} +
\delta_1 + \delta_2)$$ pulls back under $j$ to $$(10 - \e) \lambda
- \delta_{irr} - \delta_{1,\{1\}} + \psi = (10-\e)\lambda -
(5\lambda - \frac12\delta_{irr}) + \psi \sim (10 - \e') \lambda -
\delta_{irr} + 2\psi$$ which agrees with the polarization on
$\bar J/\!\!/SL_5$ for suitable $m \gg 0$.
\end{proof}

In view of the isomorphism $Y \simeq \bar M_{2,1}^{hs}$, the
discussion preceding Proposition~\ref{P:Y} implies that  the pull
back of $K_{\bar M_4^{hs}} + \frac23 \d^{hs}$ by $j'$ is nef,
giving rise to a birational morphism
  contracting the proper transform of $\bar W^1_2$.

\subsection{Proof of Theorem~\ref{T:1}}\label{S:proof}
Recall from \S\ref{S:isom} that the positive rational linear
combination $K_{\Mhs} + \a \, \dhs$ can be written as
\[
a ( K_{\Mhs} + (7/10 - \e) \dhs ) + b (K_{\Mhs} + 2/3 \, \dhs)
\]
for some positive rational numbers $a$ and $b$, and a small
positive rational number $\e$.
The first term is ample by
\cite{HH2} and the second term is semiample by our analysis in this section thus far.
 It follows that the sum is
ample and Theorem~\ref{T:1}.(1) is established.

We shall now consider the exceptional locus of the  birational
morphism associated to $|\ell(K_{\Mhs} + 2/3\, \dhs)|$ for
sufficiently large and divisible $\ell$. We have seen in
\S\ref{S:isom} that the morphism is an isomorphism away from
$\dhs_2\cup H'$. On $H'$,  $K_{\Mhs} + 2/3\, \dhs$ contracts
$\til{B}_5 = \dhs_2 \cap H'$  \cite{HL3}. Therefore, the
exceptional locus is completely contained in $\dhs_2$.
Consider the pullback of $K_{\Mhs} + 2/3\, \dhs$ by $\M^{hs}_{2,1}\times \M^{hs}_{2,1} \to \dhs_2 \inj \Mhs$. Since   $K_{\Mhs} + 2/3\, \dhs$ is nef on $\M^{hs}_{2,1}\times \M^{hs}_{2,1} $ and a curve in $
\bar M_{2,1}^{hs} \times  \bar M_{2,1}^{hs}$  is numerically
equivalent to a sum of curves from each factor by Lemma~\ref{L:slice}, it follows that $K_{\Mhs} + 2/3\, \dhs$ has zero intersection with a curve in $\M^{hs}_{2,1}\times \M^{hs}_{2,1} $ if and only if it has zero intersection with the curves from each factor. But on $\bar M_{2,1}^{hs}$,  we proved in the
previous section that $K_{\Mhs} + 2/3\, \dhs$ gives rise to a
birational morphism with exceptional locus the proper transform
 of $\bar W_2^1$.  Transferring this result to $\dhs_2$, we conclude that  the exceptional locus is precisely the locus of Weierstrass genus two tails.

\begin{remark} Since $K_{\Mhs} + 2/3\, \dhs$ pulls back to $D$ which is extremal, it follows that $K_{\Mhs} + \a \, \dhs$ is not nef on $\Mhs$ for $\a < 2/3$.  Also, it is shown in \cite{HL3} that  $K_{\Mhs} + 2/3\, \dhs$ restricts to an extremal divisor on $\overline H_4^{hs}$.
\end{remark}

\section{Moduli space of h-stable pointed curves of genus two}\label{S:pointed-hs}
In this section, we work out the GIT of bi-log-canonically embedded
pointed curves discussed in \S\ref{S:h-stable}. We retain the
notations $d, N, P(m), \hilb, \bar J$ and $\mcl L_m$ from
\S\ref{S:h-stable}. GIT of $\bar J$ gives rise to the Hilbert
semistable bi-log-canonical pointed curves. Alternatively, one can
employ the Chow variety in place of $\hilb$. Let $\chow$ denote the
Chow variety of curves of degree $d$ in $\bP^N$. Given a
bi-log-canonical model of a pointed curve $(C, p)$, we call the
corresponding point in $\chow \times \bP^N$ the {\it Chow point} of
$(C, p)$ and denote it by $Ch(C,p)$. Then we define $J' \subset
\chow \times \bP^N$ in a similar manner and consider the GIT of
$\bar J'$, linearized by (the restriction to $\bar J'$ of)
\[
\cO_{\chow}(3)\boxtimes \cO_{\bP^N}(4)
\]
which pulls back by the Chow cycle map to the leading coefficient
of the balanced linearization $\mcl L_m$ (\S\ref{S:h-stable},
Equation~(\ref{E:balanced})) on the Hilbert scheme (see
\S\ref{S:unstable} below). We say that $(C,p)$ is {\it Chow
$($semi$)$stable} if its Chow point is GIT (semi)stable. Swinarski
completely works out the theory for pointed curves $(C, p_1,
\dots, p_n)$ embedded by $(\omega_C(\sum p_i))^\nu$ for $\nu \ge
5$, and proves that such a curve is Hilbert stable if and only if
it is Deligne-Mumford stable. This result may also be derived from
\cite{BS}. We shall work out the case $g = 2$, $n = 1$ and $\nu =
2$.

\begin{definition} A pointed curve $(C,p)$ of genus two is said to be {\it c-semistable}
if
\begin{enumerate}
\item $C$ has nodes, ordinary cusps and tacnodes as
    singularity;
\item A rational component of $C$  has $\ge 3$ special points
    counting multiplicity;
\item A genus one subcurve meets the rest of the curve in $\ge
    2$ points without counting multiplicity;
\item $p$ is smooth.
\end{enumerate}
It is {\it c-stable} if it is c-semistable and $C$ has no tacnode and no subcurve of genus one.
\end{definition}

\begin{theorem} \label{T:Chow}
A bi-log-canonical pointed curve $(C, p)$ of genus two is Chow
(semi)stable if and only if it is c-(semi)stable.
\end{theorem}

The analogous statement for the Hilbert point of $(C, p)$  is given in Theorem~\ref{T:hs}. There is no strictly semistable point in the Hilbert quotient.
Therefore, the quotient space $\bar M_{2,1}^{hs} := \bar
J/\!\!/SL_{5}$ has only finite quotient singularities, and the
corresponding moduli stack $\bar {\mcl M}_{2,1}^{hs}$ is
Deligne-Mumford.



\subsection{Unstable pointed curves}\label{S:unstable}
Let $\mcl L'_\ell$ denote
\[
p_1^*\cO_{\chow}(+1)\otimes p_2^*\cO_{\bP^N}(\ell)
\]
on $\chow \times \bP^N$. Due to \cite[Theorem~4]{Kn}, we have
\begin{equation}\label{E:pullback}
\bigwedge \pi_*\cO_{\mcl C}(m) = \binom{m}{2}Ch^*\cO_{\chow}(+1) + \mbox{lower degree terms}
\end{equation}
Hence $\mcl L'_{4/3}$ pulls back via $Ch \times 1_{\bP^N}$ to a
positive rational multiple of the top coefficient of the
linearization
 $3p_1^*(\bigwedge \pi_*\cO_{\mcl C}(m)) + 2m^2
 p_2^*\cO_{\bP^N}(1)$, where $Ch: \hilb \to \chow$ is the Chow cycle map.
 It follows that
\begin{lemma}\label{L:ChowHilb} $(C, p)$ is Hilbert unstable if its Chow point
is unstable with respect to $\mcl L'_{4/3}$.
\end{lemma}

We shall employ this lemma extensively to deduce the geometric invariant theoretic property of
Hilbert points from the  GIT of corresponding Chow points,  and vice
versa.

\begin{remark} We shall interchangeably use 1-PS $\rho$ of $GL_5$ with integral weights $(r_0, \dots, r_4)$ and its corresponding 1-PS of $SL_5$ with rational weights $(r_0 - \frac15 \sum r_i, \dots, r_4 - \frac15\sum r_i)$. Also, note that
\[
\mu^{\mcl L'_{4/3}}(Ch(C, p),\rho) = \mu^{\cO_{\chow}(+1)}(Ch(C),\rho) + \frac43 \mu^{\cO_{\bP^N}(+1)}(p, \rho).
\]
Recall that the second term $\mu^{\cO_{\bP^N}(+1)}(p, \rho)$ is simply the maximum of the negative of the weights of $\rho$ on the nonzero coordinate of $p$.
\end{remark}

\begin{lemma}\label{L:smoothpoint} $(C, p)$ is Chow unstable if $p$ is not a smooth point.
\end{lemma}
\begin{proof}
Choose coordinates so that $p = [1, 0, \dots, 0]$ and let $\rho$
be a 1-PS with weights $(1, 0, \dots, 0)$. Let $\nu : \tilde{C}
\to C$ be the normalization and let $p', p'' \in \tilde{C}$ be the
points ($p' = p''$ if $p$ is a  cusp) over $p$. Then by
\cite[Lemma~1.4]{Sch}, we have
\[
e_\rho(C) \ge e_\rho(\tilde C)_{p'} + e_\rho(\tilde C)_{p"} \ge 2 \cdot 1^2
\]
and hence
\[
\mu(Ch(C), \rho) \le -2 + \frac{2d}{N+1} = -2 + \frac{2\cdot 6}5 = \frac25.
\]
 Hence
\[
\mu(Ch(C, p), \rho) \le \frac25 + \frac43\cdot (-\frac45) = -\frac23.
\]
where $-4/5 = -(1-1/5)$ is the maximum of the negative of the weights on the nonzero coordinates of $p$.
\end{proof}

\begin{lemma} If $C$ has a triple point, then $(C,p)$ is Chow unstable.
\end{lemma}
\begin{proof} We follow the proof of \cite[Proposition~3.1]{Mum}.  Choose coordinates so that $[1, 0, \dots, 0]$ is a triple point, and let $\rho$ be the 1-PS of $GL_{N+1}$ with weights $(1, 0, \dots, 0)$.  Then $\mu(Ch(C), \rho) \le -3 + \frac{2d}{N+1}$ and $\mu(p, \rho) \le \frac1{N+1}$. Hence the Hilbert-Mumford index of the Chow point of $(C,p)$ satisfies
\[
\mu(Ch(C,p), \rho) \le -3 + \frac{2d}{N+1} + \frac43 \frac1{N+1} = -3 + \frac83 < 0.
\]
\end{proof}

\begin{lemma} If $C$ has a multiple component, $(C, p)$ is Chow unstable.
\end{lemma}

\begin{proof} In \cite[Lemma~7.4]{HH2}, it is shown that there is a 1-PS $\rho$ with weights
$(3,2,1,0,0)$ such that
\[
\mu(Ch(C), \rho) \le -18 + \frac{2d}{N+1}(3+2+1)
 \]
 which in our case is equal to $-18 + \frac{12}56 = -\frac{18}5.$
The maximum of the negative of the weights is $\frac65$, and we
obtain
\[
\mu(Ch(C,p), \rho) \le -\frac{18}5 + \frac43\frac65 = -2.
\]
\end{proof}

Note that a reduced curve of genus two cannot have a singularity of the form
$y^2 = x^{m}$ for $m \ge 6$, for any such singularity would
contribute $\ge 3$ to the arithmetic genus.

\begin{lemma}
If $C$ has a rhamphoid cusp i.e. a singularity of the form $y^2 =
x^5$, then $(C, p)$ is Chow unstable.
\end{lemma}

\begin{proof}
Let $p \in C$ be a rhamphoid cusp.  Note
that the analysis in \cite[Lemma~7.2]{HH2} does not depend on the
particular embedding but only on the degree. Applying it in
our situation, we obtain a 1-PS $\rho$ with weights $(5, 3, 1, 0,
0)$ with respect to suitable coordinates, so that
\[
\mu(Ch(C), \rho) \le  -25 + \frac{2d}{N+1}(5+3+1) = -25 + \frac{108}5 = -\frac{17}5.
\]
The maximum of the negative of the weights is $\frac95$. Hence
\[
\mu(Ch(C,p), \rho) \le -\frac{17}5 + \frac43 \cdot \frac95  = -1.
\]
\end{proof}

\begin{definition}
An {\it elliptic bridge} is a pointed curve of the form $(C := E
\cup_{q_0,q_1} R, p)$ where $E$ is an elliptic curve, $R$ is a
rational curve meeting $E$ in two nodes $q_0$ and $q_1$, and the
marked point $p$ is a smooth point of $R$.
\end{definition}

\begin{proposition}\label{P:eb-special}
Let $(C^\star = E\cup_{q_0,q_1}R, p)$ be a bi-log-canonical elliptic
bridge such that $E = R_0 \cup_y R_1$ consists of two rational
curves meeting in one tacnode $y$ (Figure~\ref{F:eb}). Then there is a 1-PS $\rho$
with respect to which the $m$th Hilbert point has the
Hilbert-Mumford index
\[
\mu^{\mcl L_{m}}((C^\star,p), \rho) = 2m^2 - 7m + 5 + \frac23m^2 (-3) = -7m + 5.
\]
In particular, $(C^\star, p)$ is Hilbert unstable with respect to $\rho$.
\end{proposition}

\begin{figure}[ht!] \labellist \small\hair 2pt
\pinlabel $0$ at 130 676
\pinlabel $0$ at 181 734
\pinlabel $0$ at 221 734
\endlabellist
\centering \includegraphics[scale=0.6]{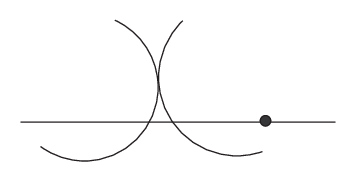}  \caption{An elliptic bridge $(C^\star, p)$}
  \label{F:eb}
\end{figure}

\begin{proof}
 Assume $q_i \in R_i$.  We have
\[
\omega_{C^\star}|R_i = \omega_{R_i}(2y + q_i) \simeq \cO_{R_i}(+1).
\]
Hence $\o_{C^\star}^{\ten 2}(2p)$ restricts to $\cO_{R_i}(+2)$ on
$R_i$. On $R$, it restricts to $\cO_R(-4 + 2q_0 +2 q_1 + 2p)
\simeq \cO_R(+2)$. Hence $C^\star$ can be parametrized by:
\[
\begin{array}{cccrrrrrrrrr}
[s_0,t_0] & \mapsto &   [&s_0^2, &s_0t_0,& t_0^2,& 0,  &   0] \\
\left[s_1,t_1\right] & \mapsto & [&0,   &  s_1t_1, &s_1^2, & t_1^2,& 0] \\
\left[s, t\right] & \mapsto &      [&s^2,   & 0,   &  0,   &  t^2, & st]
\end{array}
\]
From this parametrization, we obtain the ideal $I_{C^\star}$ of
$C^{\star}$ embedded in $\mathbb P^4$ by $\o_{C^\star}^{\ten
2}(2p)$. Let $\rho : \bC^* \to \SL_5$ with weights $(3,-2,-7,3,3)$
which comes from automorphisms of $C^\star$. Using the Gr\"obner
basis algorithm of \cite{HHL}, we obtain the filtered Hilbert
function of $C^\star$ with respect to $\rho$ as
\[
\mu([C^\star]_m, \rho) = 2 m^2 - 7 m + 5.
\]
On the other hand, the contribution from the marked point is $-3$.
With respect to $\mcl L_m$, the
Hilbert-Mumford index is
\begin{equation}\label{E:eb-hilb}
\mu^{\mcl L_{m}}((C^\star,p), \rho) = 2m^2 - 7m + 5 + \frac23m^2 (-3) = -7m + 5.
\end{equation}
\end{proof}

Now, we shall consider the basin of attraction (see \cite{HH2} for the definition) of the Hilbert point of $(C^\star, p)$ with respect to $\rho$. By definition, any curve in the basin of attraction has the same Hilbert-Mumford index with respect to $\rho$ as $(C^\star, p)$. Basins of attraction in a Hilbert scheme may be computed by using deformation theory (see, for instance, \cite[Section~5.1]{AH}).

\begin{corollary}\label{C:eb-hilb} Bi-log-canonical elliptic bridges are Hilbert unstable.
\end{corollary}
\begin{proof}
Consider the basin of attraction of the bi-log-canonical elliptic bridge $C^\star=R_0\cup R_1\cup R$  with
respect to $\rho$, where $C^\star$ and $\rho$ are as in Proposition~\ref{P:eb-special}. The local versal deformation space of a tacnode
is given by $y^2 = x^4 + a x^2 + b x + c$ where $x$ parametrizes the tangent space of (the branches at) the tacnode and $y$ is a local parameter of the tacnode (at a branch).  The tangent space of (the branches at) the tacnode $[0,0,1,0,0]$
is parametrized by $x_1/x_2$, on which $\rho$ acts with weight
$-2-(-7) = +5$.
Hence the $\rho$ action on the local versal
deformation space of the tacnode has weights $(+10, +15, +20)$. On
the other hand, at the node $q_0 = [1,0,0,0,0]$, $\rho$ acts on
the two tangent directions $x_1/x_0$ and $x_4/x_0$ with weights
$-5$ and $0$ respectively. Hence it acts on the local versal deformation
space of $q_0$ with weight $-5$. It follows that the basin of
attraction with respect to $\rho$ contains an arbitrary smoothing of the tacnode but no smoothing of the nodes.
\end{proof}

\begin{corollary}\label{C:eb-Chow} \begin{enumerate}
\item Elliptic bridges and the corresponding tacnodal pointed
    curves are Chow strictly semistable with respect to $\rho^{\pm
    1}$;
\item The basin of attraction of the Chow point of $(C^\star, p)$
    with respect to $\rho$ (resp. $\rho^{-1}$)
 contain all elliptic bridges (resp. c-semistable pointed curves with a tacnode).
 \end{enumerate}
\end{corollary}

\begin{proof}  (\ref{E:pullback}) and (\ref{E:eb-hilb}) imply that $\mu^{\mcl L_{4/3}'}((C^\star,p),\rho^{\pm 1}) = 0$. We have already analyzed the basin of attraction
with respect to $\rho$ in Corollary~\ref{C:eb-hilb}. Since $\rho^{-1}$
acts on the deformation spaces with the weights of opposite signs,
we conclude that the basin with respect to $\rho^{-1}$ contain all
smoothings of the node and no smoothing of the tacnode.
\end{proof}
Note that this completely classifies c-semistable pointed curves
that are attracted to $(C^\star, p)$ since any 1-PS coming from
the automorphism group of $(C^\star,p)$ is an integral power of
$\rho$.

\begin{lemma}\label{L:onetac} Let $C^* = E \cup_r R$ be a bi-log-canonical curve consisting of a rational cuspidal curve $E$ and a smooth rational curve $R$ meeting in a tacnode $r$. Then for any point $p \in R$, $(C^*, p)$ is Chow unstable.
\end{lemma}

\begin{proof}
Restricting $\o_C^{\ten
2}(2p)$ we get
\[
\o_C^{\ten 2}(2p)|E \simeq \cO_E(4r), \quad \o_C^{\ten 2}(2p)|R \simeq
\cO_R(-4)(2p+4r) \simeq \cO_R(2).
\]
So $E$ and $R$ can be
parametrized by
\[
[s, t] \mapsto [s^4, s^2t^2, st^3, t^4, 0]
\] and
\[
[u,v] \mapsto [0, 0, uv, u^2, v^2].
\]
The cusp is at $q=[1,0, \dots, 0]$ and the tacnode, at $r = [0,0,0,1,0]$.
Let $\rho$ be the 1-PS with weights $(0, 2, 3, 4, 2)$. Using the parametrization, we easily find the ideal of $C^*$ and the Hilbert-Mumford index $\mu([C]_m, \rho)$ by employing the Gr\"obner basis algorithm of \cite{HHL}. We have
\[
\mu([C]_m, \rho) = - 4m^2 + 4m
\]
and
\[
\mu(Ch(C), \rho) = \lim_{m\to \infty} \frac1{m^2} \mu([C]_m, \rho) = -4.
\]
The marked point $p$ contributes $\le 11/5$ to the Hilbert-Mumford index, and we find
that
\[
\mu(Ch(C^*, p), \rho) \le -4 + \frac43\frac{11}5 = -\frac{16}{15} < 0.
\]
\end{proof}

\begin{lemma} If $C$ has a genus one subcurve meeting the rest in one tacnode, $(C, p)$ is Chow unstable for any $p \in C$.
\end{lemma}

\begin{proof}
This follows since $C$ is in the basin of attraction of $C^*$ from Lemma~\ref{L:onetac}: Consider the  cusp $q$ whose local equation is $(x_2/x_0)^2 = (x_1/x_0)^3$. Its local versal deformation space is defined by $(x_2/x_0)^2 = (x_1/x_0)^3 + a (x_1/x_0) + b$, so $\bG_m$ acts on it via $\rho$ with weights $(+4,+6)$. Thus the basin of attraction of $[C]_m$ (and of $Ch(C)$)  contains arbitrary smoothing of the cusp.
Consequently, $Ch(C,p)$ is in the basin of attraction of $Ch(C^*,p^*)$ where $p^* = \lim_{\a\rightarrow 0}\rho(\a).p$. If $p$ is  not on the smooth rational component, then $(C,p)$ is Chow unstable by Propositions~\ref{P:balance} and \ref{P:limit}. So we assume that $p$  is on the smooth rational component of $C$, which forces $p^*$ to be on the smooth rational component $R$ of $C^*$. In this case we have
\[
\mu(Ch((C, p)), \rho)  = \mu(Ch((C^*, p^*)), \rho) < 0
\]
by Lemma~\ref{L:onetac}.

\end{proof}

\subsection{Proof of semistability}
It all starts with the following crucial theorem due to David Swinarski:
\begin{theorem}\cite{Swin}  A bi-log-canonically embedded $n$-pointed curve $(C, p_1, \dots, p_n)$ is Hilbert stable (and hence Chow stable)
if $C$ is smooth and $p_i$'s are distinct smooth points of $C$.
\end{theorem}

To get things rolling, we also need a pointed version of
\cite[Proposition~5.5]{Mum}:
\begin{proposition}\label{P:balance}
Let $C \subset \bP^5$ be a genus two curve of degree six and $p
\in C$ be a point such that $(C, p)$ is Chow semistable. Then for
any subcurve $C_1$ of $C$, we have
\[
| \deg C_1 - 2 \, \deg_{C_1}(\o_C(p)) | \le \frac w2
\]
where $w = \#(C_1 \cap \overline{C-C_1})$.
\end{proposition}

\begin{proof}
Following Mumford, let $L_1 = \{ x_{n_1+1} = \cdots = x_5 = 0 \}$
be the smallest linear subspace containing $C_1$, and let $\rho$
denote the 1-PS defined by
\[
\rho(t).x_i = \begin{cases}
t\, x_i, \quad i \le n_1 \\
x_i, \quad i > n_1.
\end{cases}
\]
By the Chow semistability of $(C, p)$, we have
\[
-e_\rho(C) + \frac{12}5(n_1+1) + \frac4{15}(n_1+1) - \frac43\delta(p) \ge 0
\]
where $\delta(p) = 0$ if $p \not\in L_1$ and $\delta(p) = 1$ if $p
\in L_1$. Passing to the partial normalization $C_1 \coprod
\overline{C-C_1}$, we have
\[
e_\rho(C) \ge w + 2 \, \deg C_1.
\]
Combining the two inequalities above, we obtain
\begin{equation}\label{E:ineq}
2 \, \deg C_1 + w \le  \frac{12}5(n_1+1) + \frac4{15}(n_1+1) - \frac43\delta(p).
\end{equation}
Since a subcurve of $C$ is of genus zero or one, it is embedded by
a non-special linear system. Thus $n_1 + 1  = \deg C_1 + 1 - g_1$
and substituting it in (\ref{E:ineq}) yields
\[
- \frac w2 \le \deg C_1 - 4 g_1 + 4 - 2w - 2\delta(p) = \deg C_1 - 2 \deg_{C_1}(\o_C(p)).
\]
Applying this inequality to $C_2 = \overline{C - C_1}$ and using
\[
0 = \deg_C - 2 \, \deg \omega_C(p) = \sum_{i=1}^2 \deg C_i - 2 \, \deg_{C_i}(\o_C(p))
\]
we obtain $\deg C_1 - 2 \, \deg_{C_1}(\o_C(p)) \le \frac w2.$
\end{proof}

\begin{proposition}\label{P:limit} Let $(\mcl C, \Sigma) \to \spec k[[t]]$ be a family of Chow semistable pointed curves of genus $g$ whose generic fibre $\mcl C_\eta$ is smooth. Here $\Sigma : \spec k[[t]] \to \mcl C$ is necessarily  a section of smooth points. If $\Phi : \mcl C \to \bP^{3g-2}_{k[[t]]}$ is an embedding such that $\Phi_\eta^*(\cO(1)) = \omega_{\mcl C_\eta/k((t))}^{\otimes 2}(2\Sigma|\mcl C_\eta)$, then $\cO(1) = \omega_{\mcl C/k[[t]]}^{\otimes 2}(2\Sigma)$.
\end{proposition}
In view of the relation between the Hilbert semistability  and the Chow semistability, we can
safely replace Chow by Hilbert in the statement.
\begin{proof} The proof is due to Mumford and the assertion follows from Proposition~\ref{P:balance}. We just point out that the extra data of a section does not effect the argument.
Let $C_i$ denote the components of $\mcl C_0$. Then $\cO(1) \simeq
\left(\o_{\mcl C/k[[t]]}^{\otimes 2}(2\Sigma)\right)(\sum r_i
C_i)$
 for some $r_i$ which
we may assume to be nonnegative integers such that $\min \{r_i\} =
0$. Let $C_+ = \cup_{r_i > 0} C_i$ and $C_o = \cup_{r_i = 0} C_i$.
Then
\[
\#(C_+ \cap C_o) \le \deg_{C_+} \cO_{\mcl C}(\sum r_i C_i)
= \deg C_+ - 2 \, \deg_{C_+} \o_{\mcl C_0}(p).
\]
which contradicts Proposition~\ref{P:balance} unless $r_i = 0$ for
all $i$.
\end{proof}

The proposition implies that if $(C, p)$ is Chow semistable
bi-log-canonical pointed curve of genus two, then
\begin{enumerate}
\item $C$ is nondegenerate and;
\item a rational component of $C$ has three special points
   counting multiplicity and;
\item $C$ has no elliptic tail.
\end{enumerate}
For instance, suppose that $C$ had an elliptic tail. Then $C = E_1
\cup_q E_2$ where $E_i$ are subcurves of genus one meeting each
other in one node $q$, and the bi-log-canonical system
$\omega_C^{\otimes 2}(2 p)$ is not very ample on $E_i$ if $p
\not\in E_i$.

\

We now complete the proof of Theorems~\ref{T:hs} and
\ref{T:Chow}. Let $(C, p)$ be a c-semistable bi-log-canonical pointed
curve. Consider a smoothing $\pi: (\mcl C, \Sigma) \to \spec
k[[t]]$ of $C$ and embed in $\bP^{3g-2}_{k[[t]]}$ by choosing a
frame for $\pi_*(\omega_{\mcl C/k[[t]]}^{\otimes 2}(2 \Sigma))$.
This induces a map $\spec k((t)) \to (\chow \times \bP^N)^{ss}$.
Applying the semistable replacement theorem, we obtain a map
$\spec k[[t]] \to (\chow \times \bP^N)^{ss}$ and the corresponding
 Chow semistable family $(\mcl D, \Sigma') \to \spec k[[t]]$ whose generic member is isomorphic to that of $(\mcl C, \Sigma)$. It remains to prove that they agree at the special fibre. Note that $\mcl D_0$ and $C$ necessarily have the same Deligne-Mumford stabilization. So $\mcl D_0  = C$ is evident if $C$ has no tacnode and no elliptic bridge as there is no other c-semistable curve that has the same Deligne-Mumford stabilization.
 Consider the elliptic bridge $(C^\star, p)$ from Proposition~\ref{P:eb-special}. The only c-semistable pointed curves  that have the same Deligne-Mumford stabilization as
 $(C^\star, p)$ are elliptic bridges and the tacnodal ones. But these are contained in the basin of attractions of $(C^\star, p)$, and we conclude that all elliptic bridges and c-semistable pointed curves with a tacnode are Chow semistable \cite[Lemma~4.3]{HH2}. It also follows that if $C$ is c-stable then it is Chow stable since otherwise $C$ would be contained in a basin of attraction of $(C^\star, p)$, the unique c-semistable pointed curve with infinite automorphisms.

Note that an h-stable pointed curve without a tacnode is c-stable
by definition and hence Chow stable. By
Lemma~\ref{L:ChowHilb}, it is Hilbert stable. Let $(C, p)$ be an
h-stable curve with a tacnode. Suppose it is Hilbert unstable.
Then there is a 1-PS $\rho'$ such that $\mu^{\mcl L_m}((C,p),
\rho') < 0$ for $m \gg 0$. Let $C_0$ denote the flat limit
$\lim_{t \rightarrow 0} \rho'(t).C$. Since $(C, p)$ is
c-semistable, $\mu^{\mcl L'_{3/4}}((C,p), \rho') = 0$, which
implies that $C_0$ is also c-semistable. But a c-semistable
pointed curve with infinite automorphisms is of the form $(C^\star, p)$ from Proposition~\ref{P:eb-special}, and it follows from Corollary~\ref{C:eb-Chow} that
$\rho'$ is a positive multiple of $\rho^{-1}$. But this is a
contradiction to Proposition~\ref{P:eb-special} which implies
\[
\mu^{\mcl L_m}((C,p), \rho^{-1}) = 7m - 5 > 0.
\]

\bibliographystyle{alpha}
\bibliography{m4}

\begin{thebibliography}{CMJL12}

\bibitem[AFS]{AFS}
Jarod Alper, Maksym Fedorchuk, and David Smyth.
\newblock Finite hilbert stability of (bi)canonical curves.
\newblock {\em Invent. Math.}, pages 1--48.
\newblock 10.1007/s00222-012-0403-6.

\bibitem[AH12]{AH}
Jarod Alper and Donghoon Hyeon.
\newblock G{IT} constructions of log canonical models of {$\overline M\sb
  {g}$}.
\newblock In {\em Compact moduli spaces and vector bundles}, volume 564 of {\em
  Contemp. Math.}, pages 87--106. Amer. Math. Soc., Providence, RI, 2012.

\bibitem[BS08]{BS}
Elizabeth Baldwin and David Swinarski.
\newblock A geometric invariant theory construction of moduli spaces of stable
  maps.
\newblock {\em Int. Math. Res. Pap. IMRP}, (1):Art. ID rpn 004, 104, 2008.

\bibitem[CMJL]{CMJL12}
Sebastian Casalaina-Martin, David Jensen, and Radu Laza.
\newblock Log canonical models and variation of git for genus four canonical
  curves.
\newblock {\it J. Algebraic Geom.} (to appear), 2013.

\bibitem[CMJL12]{CMJL}
Sebastian Casalaina-Martin, David Jensen, and Radu Laza.
\newblock The geometry of the ball quotient model of the moduli space of genus
  four curves.
\newblock In {\em Compact moduli spaces and vector bundles}, volume 564 of {\em
  Contemp. Math.}, pages 107--136. Amer. Math. Soc., Providence, RI, 2012.

\bibitem[DM69]{DM}
Pierre Deligne and David Mumford.
\newblock The irreducibility of the space of curves of given genus.
\newblock {\em Inst. Hautes \'Etudes Sci. Publ. Math.}, (36):75--109, 1969.

\bibitem[Fab99]{faber-1997}
Carel Faber.
\newblock Algorithms for computing intersection numbers on moduli spaces of
  curves, with an application to the class of the locus of {J}acobians.
\newblock In {\em New trends in algebraic geometry ({W}arwick, 1996)}, volume
  264 of {\em London Math. Soc. Lecture Note Ser.}, pages 93--109. Cambridge
  Univ. Press, Cambridge, 1999.

\bibitem[Fed]{Fedorchuk}
Maksym Fedorchuk.
\newblock The final log canonical model of the moduli space of stable curves of
  genus 4.
\newblock {\it Internat. Math. Res. Notices} (to appear), 2012.

\bibitem[FJ]{FJ}
Maksym Fedorchuk and David Jensen.
\newblock Stability of 2nd {H}ilbert points of canonical curves.
\newblock {\it Internat. Math. Res. Notices} (to appear), 2012.

\bibitem[GKM02]{GKM}
Angela Gibney, Sean Keel, and Ian Morrison.
\newblock Towards the ample cone of {$\overline M\sb {g,n}$}.
\newblock {\em J. Amer. Math. Soc.}, 15(2):273--294 (electronic), 2002.

\bibitem[Has05]{Has}
Brendan Hassett.
\newblock Classical and minimal models of the moduli space of curves of genus
  two.
\newblock In {\em Geometric methods in algebra and number theory}, volume 235
  of {\em Progr. Math.}, pages 169--192. Birkh\"auser Boston, Boston, MA, 2005.

\bibitem[HH]{HH2}
Brendan Hassett and Donghoon Hyeon.
\newblock Log minimal model program for the moduli space of curves: the first
  flip.
\newblock {\it Ann. of Math.} (to appear), 2013.

\bibitem[HH09]{HH1}
Brendan Hassett and Donghoon Hyeon.
\newblock Log canonical models for the moduli space of curves: the first
  divisorial contraction.
\newblock {\em Trans. Amer. Math. Soc.}, 361(8):4471--4489, 2009.

\bibitem[HHL10]{HHL}
Brendan Hassett, Donghoon Hyeon, and Yongnam Lee.
\newblock Stability computation via {G}r\"obner basis.
\newblock {\em J. Korean Math. Soc.}, 47(1):41--62, 2010.

\bibitem[HL10a]{HL2}
Donghoon Hyeon and Yongnam Lee.
\newblock Log minimal model program for the moduli space of stable curves of
  genus three.
\newblock {\em Math. Res. Lett.}, 17(4):625--636, 2010.

\bibitem[HL10b]{HL3}
Donghoon Hyeon and Yongnam Lee.
\newblock A new look at the moduli space of stable hyperelliptic curves.
\newblock {\em Math. Z.}, 264(2):317--326, 2010.

\bibitem[Kim08]{Kim}
Hosung Kim.
\newblock Chow stability of curve of genus 4 in $\mathbb{P}^3$, 2008.
\newblock Ph. D. Thesis, Sogang University.

\bibitem[KM76]{Kn}
Finn~Faye Knudsen and David Mumford.
\newblock The projectivity of the moduli space of stable curves. {I}.
  {P}reliminaries on ``det'' and ``{D}iv''.
\newblock {\em Math. Scand.}, 39(1):19--55, 1976.

\bibitem[KM98]{KM}
J{\'a}nos Koll{\'a}r and Shigefumi Mori.
\newblock {\em Birational geometry of algebraic varieties}, volume 134 of {\em
  Cambridge Tracts in Mathematics}.
\newblock Cambridge University Press, Cambridge, 1998.
\newblock With the collaboration of C. H. Clemens and A. Corti, Translated from
  the 1998 Japanese original.

\bibitem[Kol97]{Kollar}
J{\'a}nos Koll{\'a}r.
\newblock Singularities of pairs.
\newblock In {\em Algebraic geometry---{S}anta {C}ruz 1995}, volume~62 of {\em
  Proc. Sympos. Pure Math.}, pages 221--287. Amer. Math. Soc., Providence, RI,
  1997.

\bibitem[MS11]{MS}
Ian Morrison and David Swinarski.
\newblock Gr\"obner techniques for low-degree {H}ilbert stability.
\newblock {\em Exp. Math.}, 20(1):34--56, 2011.

\bibitem[Mum77]{Mum}
David Mumford.
\newblock Stability of projective varieties.
\newblock {\em Enseignement Math. (2)}, 23(1-2):39--110, 1977.

\bibitem[Rul01]{Rulla}
William Rulla.
\newblock The birational geometry of $\bar {M}_3$ and $\bar {M}_{2,1}$, 2001.
\newblock Ph. D. thesis, University of Texas at Austin.

\bibitem[Sch91]{Sch}
David Schubert.
\newblock A new compactification of the moduli space of curves.
\newblock {\em Compositio Math.}, 78(3):297--313, 1991.

\bibitem[Sim]{Simpson}
Matthew Simpson.
\newblock On log canonical models of the moduli space of stable pointed curves.
\newblock arXiv:0709.4037v1 [math.AG].

\bibitem[Swi08]{Swin}
David Swinarski.
\newblock Geometric invariant theory and moduli spaces of pointed curves.
\newblock {\em Thesis, Columbia University}, 2008.

\end{thebibliography}

\end{document}